\documentclass[reqno,11pt]{amsart}

\usepackage{mathrsfs, amssymb, latexsym, amsmath, graphicx, psfrag}

\newtheorem{thm}{Theorem}[section]
\newtheorem{prop}[thm]{Proposition}

\newtheorem{lemma}[thm]{Lemma}

\theoremstyle{remark}
\newtheorem{rmk}{Remark}[section]

\renewcommand{\le}{\leqslant}
\newcommand{\la}{\langle}
\newcommand{\ra}{\rangle}

\newcommand{\tz}{\tau_\zeta}
\newcommand{\iz}{\int_\zeta}
\newcommand{\izpm}{\int_{\zeta,\pm}}
\newcommand{\zz}{Z_\zeta}

\newcommand{\pz}{\psi_\zeta}
\newcommand{\jz}{\Gamma_{\!\zeta}}
\newcommand{\uz}{U_{\!\zeta}}
\newcommand{\uzb}{\bar \uz}
\newcommand{\h}{{\mathrm h}}
\newcommand{\e}[1]{\epsilon_{#1}}
\newcommand{\id}{\mathrm{Id}}
\newcommand{\tr}{\mathrm{tr}}
\newcommand{\trz}{\mathrm{tr}_q^\omega}
\newcommand{\trze}[1]{\mathrm{Tr}_{\zeta,-\e {#1}}}
\newcommand{\ize}[1]{\int_{\zeta,-\e {#1}}}

\newcommand{\trzpm}{\mathrm{Tr}_{\zeta,\pm}}
\newcommand{\R}{{\mathscr R}}
\newcommand{\D}{{\mathscr D}}
\newcommand{\Dz}{\D_\zeta}
\newcommand{\Rz}{\R_\zeta}
\newcommand{\Tz}{\Theta_\zeta}

\newcommand{\br}{{\bf r}}

\newcommand{\bbz}{\mathbb{Z}}
\newcommand{\bbr}{\mathbb{R}}
\newcommand{\bbq}{\mathbb{Q}}
\newcommand{\fsl}[1]{\mathfrak{s}\mathfrak{l}_{#1}}
\newcommand{\ve}[1]{V_{#1}}
\newcommand{\trir}{\vartriangleright}
\newcommand{\A}{\mathscr{A}}
\newcommand{\U}{\mathscr{U}}
\newcommand{\uh}{U_h}
\newcommand{\tu}{\tilde\U}
\renewcommand{\tt}{\tilde\otimes}
\newcommand{\gh}{\Gamma_h}
\newcommand{\Inv}{\mathrm{Inv}}

\title[Relation between quantum invariants]{On the relation between the WRT invariant and the Hennings invariant}
\author{Qi Chen,\ \  Srikanth Kuppum\ \  and\ \  Parthasarathy Srinivasan}

\address{Department of Mathematics, Winston-Salem State University\newline\indent Winston Salem, NC 27110, USA}
\email{chenqi@wssu.edu}
\address{Department of Mathematics,
Indian Institute of Technology Guwahati\newline\indent 
Guwahati 781039, Assam, India}
\email{kvsrikanth@iitg.ernet.in}
\address{Mathematical Biosciences Institute, The Ohio State University\newline\indent  Columbus, OH 43210, USA}
\email{psrinivasan@mbi.osu.edu}

\begin{document}

\begin{abstract}
The purpose of this note is to provide a simple relation between the Witten-Reshetikhin-Turaev $SO(3)$ invariant and the Hennings invariant associated to quantum $\fsl2$.
\end{abstract}

\maketitle
\addtocounter{footnote}{1}

\section{Introduction}
For any quantum group associated to a finite dimensional Lie algebra, there are two ways to obtain invariants of closed oriented 3-manifolds. The first one, predicted by Witten and rigorously constructed by Reshetikhin and Turaev, uses the representation theory of the quantum group at roots of unity. The second one, introduced by Hennings in \cite{hennings}, uses the integral of the quantum group at roots of unity. While the first one has been intensively studied because of its connection to various areas of mathematics and physics, the second one receives relatively low attention. In this note we will show that despite very different origins they are essentially the same for rational homology 3-spheres.

For any 3-manifold $M$, denote by $\h(M)$ the order of $H_1(M)$ if it is finite and $0$ otherwise. 
For any complex root of unity $\zeta$ of odd order, let $\tz(M)$ be the Witten-Reshetikhin-Turaev $SO(3)$ invariant (or WRT invariant for short),  defined first in \cite{km}, and $\pz(M)$ be the Hennings invariant associated to quantum $\fsl2$ at $\zeta$. These two invariants are related by the following theorem.

\begin{thm}\label{thm1}
If $\zeta$ is a complex root of unity of odd order $>1$ then
\begin{equation}\label{eqn1}
\pz(M) = \h(M)\, \tz(M).
\end{equation}
\end{thm}

The first indication that there might be such a relation is stated in Remark 2.11 of \cite{o3}. It is formulated explicitly in \cite{kerler4} as Conjecture 18. See also Problem 8.18 (1) of \cite{o5}. The special case, when $M$ is a lens space and the order of $\zeta$ is an odd prime, is proved in \cite{kerler3} Corollary 16.

\begin{rmk}
Kerler analyzed the center of quantum $\fsl2$ at any odd root of unity in \cite{kerler1}. The proof of Theorem \ref{thm1} is based on this result. Recently Feigin et al. gave similar analysis on the center of quantum $\fsl2$
at any even root of unity in \cite{feigin}. (We thank the referee for pointing out this reference to us.) One would guess that the same proof given in this paper should work for the even root case. But it turns out that an equation similar to (\ref{ex}) in this case is not true.  
\end{rmk}

\begin{rmk}
The WRT invariant and the Hennings invariant can be defined for higher rank quantum groups.
In that case (\ref{eqn1}) should still be true if we raise $\h(M)$ to a certain power. 
We will discuss this in a different paper.
\end{rmk}

\noindent{\em Acknowledgment}. The first author would like to thank L. Kauffman, T. Le, T. Ohtsuki, D. Radford and especially T. Kerler for helpful discussions. The authors also would like to thank the referee for many valuable suggestions.

\section{The WRT and the Hennings invariants}
In this section we will first review quantum $\fsl2$ at roots of unity and then give some quick account for the
associated WRT and Hennings invariants. Both invariants can be conveniently calculated through the universal quantum $\fsl2$ invariant of links.
 
\subsection{\boldmath Quantum $\fsl2$ at roots of unity}\label{secsl2}
Let $\ell>1$ be a positive {\em odd} integer and $\zeta$ be a primitive $\ell$-th root of unity. The quantum $\fsl2$ at $\zeta$, denoted $\uz$, is an algebra over $\bbq(\zeta)$ generated by $K, E$ and $F$ with relations:
$$
K^\ell = 1, \quad E^\ell = F^\ell = 0,
$$
$$
K E = \zeta^2 E K, \quad K F = \zeta^{-2} F K, \quad E F - F E = \frac{K-K^{-1}}{\zeta - \zeta^{-1}}.
$$
It can be given a ribbon Hopf algebra structure with the ribbon element $\br$ (see (\ref{r})) and the universal $R$-matrix
\begin{equation}\label{e:R}
\Rz = \Tz \Dz,
\end{equation}
where
\begin{equation}\label{tz}
\Tz = \sum_{m=0}^{\ell-1} \frac{(\zeta^{-1}-\zeta)^m}{\zeta^{\frac{m(m-1)}2} [m]!} F^m\otimes E^m
\end{equation}
is the so called quasi-$R$-matrix and
\begin{equation}\label{e:D}
\Dz = \frac1\ell\sum_{0\le i, j \le \ell-1} \zeta^{2ij} K^i\otimes K^j
\end{equation}
is the diagonal part of $\Rz$. Here the quantum integer is denoted $[j] = \frac{\zeta^j - \zeta^{-j}}{\zeta - \zeta^{-1}}$. One has
\begin{align}
(E^i\otimes 1) \Dz = \Dz (E^i\otimes K^i), \quad &
(1\otimes E^i) \Dz = \Dz (K^i\otimes E^i),  \label{e:sD1}\\
(F^i\otimes 1) \Dz = \Dz (F^i\otimes K^{-i}), \quad &
(1\otimes F^i) \Dz = \Dz (K^{-i}\otimes F^i). \label{e:sD2}
\end{align}

Next we follow \cite{kerler1} to describe the center $Z(\uz)$ of $\uz$. First of all it contains the Casimir element
\begin{equation*}
C = (\zeta - \zeta^{-1}) FE + \frac{\zeta K + \zeta^{-1} K^{-1}}{\zeta - \zeta^{-1}}.
\end{equation*}
Let $b_j = (\zeta^{2j+1}+\zeta^{-2j-1})/(\zeta-\zeta^{-1})$ and
\begin{equation}\label{e:min}
\phi(x) = \prod_{i=0}^{\ell-1} (x-b_i).
\end{equation}
Set
$$
\phi_j(x) = \prod_{\substack{0\le i \le \ell-1 \\ b_i\ne b_j}} (x-b_i), \quad 0\le j \le \hbar,
$$
where $\hbar = (\ell-1)/2$. Since $b_j = b_{\ell-1-j}$, $\deg(\phi_j) = \ell-2$ for $1\le j<\hbar$ and
$\deg(\phi_\hbar) = \ell-1$. Let
\begin{equation}\label{P}
P_j = \frac1{\phi_j(b_j)}\phi_j(C) - \frac{\phi_j'(b_j)}{\phi_j(b_j)^2} (C-b_j)\phi_j(C), \quad 0\le j \le \hbar,
\end{equation}
and
\begin{equation}\label{N}
N_j = \frac1{\phi_j(b_j)}(C-b_j)\phi_j(C), \quad 0\le j \le \hbar-1.
\end{equation}
To describe the other elements in $Z(\uz)$, we introduce the following polynomials of $K$:
$$
\pi_{\! j} = \frac1\ell \sum_{i=1}^\ell \zeta^{2ij}K^i, \quad 0\le j \le \ell-1,
$$
and
$$
T_j = \sum_{i=j+1}^{\ell-1-j} \pi_i, \quad 0\le j \le \hbar-1.
$$
Let
\begin{equation}\label{N'}
N_j' = T_j N_j, \quad 0\le j \le \hbar-1.
\end{equation}
It was shown in \cite{kerler1}, where $N_j'$ is denoted $N_j^+$,
that $Z(\uz)$ is a $(3\hbar+1)$-dimensional vector space with basis
\begin{equation}\label{zb}
\{P_i, N_j, N_j', \quad 0\le i \le \hbar, \quad 0\le j \le \hbar-1\}.
\end{equation}
The basis elements satisfy:
$$
P_i P_j = \delta_{ij} P_i, \quad P_i N_j = \delta_{ij} N_j, \quad P_i N_j' = \delta_{ij} N_j'
$$
and
$$
N_i N_j = N_i' N_j = N_i' N_j' = 0.
$$
To simplify expressions, from now on we consider $N_\hbar=N_\hbar'=0$ and use the convention
$$
N_\hbar/[\ell] = N_\hbar'/[\ell] = 0
$$
although $[\ell]=0$. The ribbon element and its inverse can be written as
\begin{equation}\label{r}
\br^{\pm 1} = \sum_{i = 0}^{\hbar} 
\zeta^{\pm 2i(i+1)}\left(P_i \pm \frac{2i+1-\ell}{[2i+1]} N_i \pm \frac\ell{[2i+1]} N_i'\right).
\end{equation}

For $i=0,1,\ldots, \ell-2$, denote by $\ve i$ the irreducible $(i+1)$-dimensional representation of $\uz$.

\subsection{\boldmath The universal quantum $\fsl2$ invariant of bottom tangles and links}\label{sss}
We follow \cite{habiro4}; also see \cite{habiro1}.
By a bottom tangle we mean
an oriented framed tangle, consisting of interval components only, embedded in $\bbr^2\times [0,1)$ 
such that the $i$-th component starts from $(0, 2i, 0)$ and ends at $(0, 2i-1,0)$. Two bottom tangles are equivalent if they are ambient isotopic relative to boundary.
For any bottom tangle, label its diagram by elements from $\uz$ 
according to Figure \ref{f3}, where $\Rz = \sum a\otimes b$ and $S$ are
the universal $R$-matrix and the antipode of $\uz$ respectively.
\begin{figure}[ht]
\centering
\psfrag{1}{\hspace{0ex}$a$}
\psfrag{2}{\hspace{0ex}$b$}
\psfrag{3}{\hspace{-2ex}$S(a)$}
\psfrag{4}{\hspace{-1ex}$b$}
\psfrag{5}{$K$}
\psfrag{6}{$K^{-1}$}
\includegraphics[width=\textwidth]{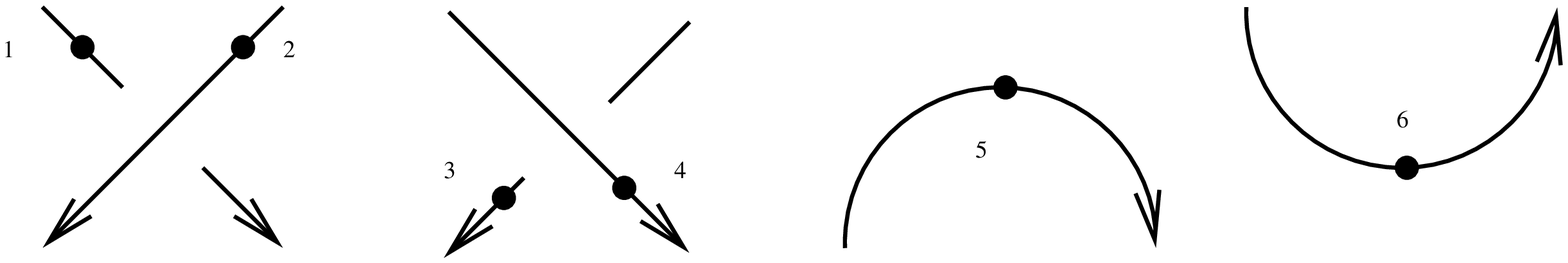}
\caption{}\label{f3}
\end{figure}
To get the universal invariant $\jz$ of bottom tangles associated to $\uz$, one follows the opposite of the orientation to multiply labels on each component. 
Hence if $T$ is a bottom tangle with $m$ components then
$\jz(T)$ belongs to $\uz^{\otimes m}$.
It is well known that if a $\pm1$ full twist is inserted to a component then the corresponding tensor factor in the universal invariant is multiplied by $\br^{\mp1}$.
In particular if $t_\pm$ is the 1-component bottom tangle with a $\pm1$ full twist then
\begin{equation}\label{ft}
\jz(t_\pm) = \br^{\mp 1}.
\end{equation} 

We want to extend the domain of $\jz$ to include framed links. 
In order to do that we need the following notion from \cite{habiro1}.
Suppose $A$ is a Hopf algebra. For any $A$-module $W$, denote by 
\begin{equation*}\label{wb}
\bar W := W/\{ x(w) - \epsilon(x)w: \forall x\in A,\ w\in W\}.
\end{equation*}
Denote by $\trir$ the left adjoint action of $A$ on itself, i.e.
$$
x\trir y = \sum_{(x)} x'y S_A(x''),\quad \forall x, y\in A,
$$
where Sweedler's notation is used for the coproduct $\Delta_A(x) = \sum_{(x)}x'\otimes x''$.
The universal quantum trace
$$
\tr_q : A \to \bar A
$$
is the canonical projection with $A$ considered as an $A$-module through $\trir$.
For any oriented framed link $L$, let $T_L$ be a bottom tangle whose canonical closure is $L$.
Set
$$
\jz(L) := \tr_q\otimes\cdots\otimes\tr_q (\jz(T_L)).
$$
It turns out that $\jz$ is an invariant of bottom tangles and links, cf. Section 7.3 of \cite{habiro1}.

\begin{rmk}\label{rk}
If $A$ is a ribbon Hopf algebra and $V$ is an $A$-module then the quantum trace $\tr_q^V(x) := \tr^V(\kappa x)$ factors
through $\bar A$. Here $\kappa$ is the associated grouplike element\footnote{The associated grouplike element of $\uz$ is $K^{-1}$.} of $A$, cf. Section 7.2 of \cite{habiro1}.
\end{rmk}

\subsection{The WRT invariant}\label{wrt}
We now briefly define the WRT $SO(3)$
invariant. Set 
$$
\omega = \omega_\zeta := \sum_{i=0}^{(\ell-3)/2} [2i+1] \ve {2i},
$$
a special element in the representation ring of $\uz$.
Suppose that $M$ can be obtained from $S^3$ by surgery along $L$. Impose an arbitrary orientation on $L$.
The WRT $SO(3)$ invariant of $M$ is
$$
\tz(M) := \frac{\trz\otimes\cdots\otimes\trz (\jz(L)).
}{\trz(\br^{-1})^{\sigma_+}\trz(\br)^{\sigma_-}},
$$ 
where $\sigma_+$ (resp. $\sigma_-$) is the number of positive (resp. negative) eigenvalues of $L$'s linking matrix.
The discrepancy of the signs in the denominator is due to (\ref{ft}).

\subsection{The Hennings invariant}\label{seccen}
The Hennings invariant is introduced in \cite{hennings} and further studied in \cite{kr, o3}. See also Section 12 of \cite{habiro1} for a review of this invariant. Instead of using the representation theory of $\uz$, the Hennings invariant makes use of a nonzero left integral for $\uz^*$. Recall that a left integral $\lambda$ for $A^*$, the dual of a finite dimensional Hopf algebra $A$, is an element in $A^*$ such that
$$
f\cdot \lambda = f(1) \lambda, \quad \forall f\ \in A^*.
$$
(Or equivalently $\lambda$ is determined by the equation
$$
(\id\otimes\lambda)\Delta(x) = 1_A\lambda(x), \quad\forall x\in A,
$$
where $\Delta$ and $1_A$ are the coproduct and the unit of $A$ respectively.)
Nonzero left integrals for $A^*$ exist and are unique up to scalar multiples.
Let $\iz = \int_{\uz}$ be the left integral of $\uz^*$ such that
\begin{equation}\label{i}
\iz F^a K^b E^c = \delta_{a,\ell-1}\delta_{b, 1} \delta_{c, \ell-1}.
\end{equation}

\begin{rmk}\label{rk1}
Suppose $A$ is unimodular, i.e. it has a left and right integral for $A$. In this case the left integrals for $A^*$ factor through $\bar A$, cf. Proposition 8 of \cite{larson}. It is easy to check that $F^{\ell-1}\pi_1 E^{\ell-1}$ is a left and right integral for $\uz$ so $\uz$ is unimodular.
\end{rmk}

Suppose $M$ can be obtained from $S^3$ by surgery along a framed link $L$. Impose an arbitrary orientation on $L$. The Hennings invariant of $M$ is
\begin{equation}\label{hh}
\pz(M) := \frac{\iz\otimes\cdots\otimes\iz\big(\jz(L)\big)}{(\iz \br^{-1})^{\sigma_+}(\iz \br)^{\sigma_-}},
\end{equation}
where $\sigma_+$ (resp. $\sigma_-$) is the number of positive (resp. negative) eigenvalues of $L$'s linking matrix.

\section{Proof of Theorem \ref{thm1}}
We first state a special case of Theorem \ref{thm1}. Recall that a framed link in $S^3$ is said to be {\em algebraically split} if its linking matrix is diagonal.

\begin{prop}\label{prop2}
Let $\zeta$ be a complex root of unity of odd order $>1$. Suppose $M$ is a rational homology 3-sphere which 
can be obtained from $S^3$ by surgery along an algebraically split link. 
Then (\ref{eqn1}) holds, i.e.
$\pz(M) = \h(M)\, \tz(M).$
\end{prop}

A proof of Proposition \ref{prop2} is given in Section \ref{ls}.
Our main result follows easily from this proposition.

\begin{proof}[Proof of Theorem \ref{thm1}]
We divide the proof into three cases.

{\em Case 1}. Suppose that $\h(M) = 0$. Ohtsuki proved in \cite{o3} that in this case $\pz(M)=0$. See also \cite{kerler2}.

{\em Case 2}. Suppose that $\h(M)\ne0$ and $M$ can be obtained from $S^3$ by surgery along 
an algebraically split link. This case is just Proposition \ref{prop2}.

{\em Case 3}. Suppose that $\h(M) \ne 0$ and $M$ can not be obtained from $S^3$ by surgery along
an algebraically split link. According to \cite{o1}, there exist lens spaces 
$L(n_1, 1), \ldots, L(n_m, 1)$ such that
$$
M' = M\# L(n_1, 1)\# \cdots\# L(n_m, 1)
$$
can be obtained from $S^3$ by surgery along an algebraically split link. By case 2,
$\pz(M') = \h(M') \tz(M')$.
As $\tz$, $\pz$ and $\h$ are multiplicative with respect to connected sum we have
$$
\pz(M) \prod_{i=1}^m\pz\big(L(n_i,1)\big) = \h(M) \tz(M) \prod_{i=1}^m |n_i| \tz\big(L(n_i,1)\big).
$$
The formulas in \cite{lili2} indicate that $\tz(L(n_i,1))\ne 0$, where $\tz$ is denoted $\tau_\ell'$. Apply case 2 again, we have
$$
\pz\big(L(n_i,1)\big) = |n_i|\, \tz\big(L(n_i,1)\big).
$$
Hence (\ref{eqn1}) follows, concluding the proof of Theorem \ref{thm1}.
\end{proof}

\section{Proof of Proposition \ref{prop2}}
A bottom tangle is said to have 0 linking matrix if its natural closure has 0 linking matrix.
For any $m$-component bottom tangle $T$ with 0 linking matrix, $\jz(T)$ is {\em a priori} an element in $\uz^{\otimes m}$. The key point of the proof of Proposition \ref{prop2} is to show that the element actually belongs to a smaller set.

\subsection{\boldmath The universal quantum $\fsl2$ invariant of bottom tangles with 0 linking matrix}
Denote by $\zz'$ the subspace of $Z(\uz)$ spanned by polynomials of the Casimir element $C$. 

\begin{lemma}\label{b}
A basis of $\zz'$ is given by $\{P_i,\ N_j,\ 0\le i \le \hbar,\ 0\le j\le \hbar-1\}$.
\end{lemma}

\begin{proof}
According to \cite{kerler1}, (\ref{e:min}) is the minimal polynomial of $C$. Hence
$\zz'$ has dimension $\ell = 2\hbar+1$. The lemma follows from the fact
that every polynomial of $C$ is a linear combination of $P_i$ and $N_j$, cf.
(3.95) of \cite{kerler1}.
\end{proof}

The following lemma is the key to our main result.

\begin{lemma}\label{cen}
Suppose $T$ is a bottom tangle with 0 linking matrix. 
Then $(\id\otimes\chi_2\otimes\cdots\otimes\chi_m)(\jz(T))$ belongs to $\zz'$ where $\chi_j: \uz\to\bbq(\zeta)$ is linear and factors through $\uzb$, $2\le j\le m$, cf. Section \ref{sss}.
\end{lemma}

Before giving a proof of this lemma, we need to define three more versions of quantum $\fsl2$.
Let $h$ be a formal variable and $\uh=\uh(\fsl2)$ be an algebra over $\bbq[[h]]$, topologically generated by
$H$, $E$ and $F$ such that
$$
HE=E(H+2),\quad HF=F(H-2), \quad EF-FE = \frac{K-K^{-1}}{q-q^{-1}},
$$
where $K=\exp(hH/2)$ and $q=\exp(h/2)$. Denote by $\uh\hat\otimes\uh$ the $h$-adic completion of $\uh\otimes\uh$.
One can give $\uh$ a ribbon Hopf algebra structure with the universal $R$-matrix
$$
\R_h = \Theta_h \D_h \quad\in \uh\hat\otimes\uh,
$$
where 
\begin{equation}\label{th}
\Theta_h=\sum_{m=0}^{\infty} \frac{(q^{-1}-q)^m}{q^{\frac{m(m-1)}2} [m]!} F^m\otimes E^m
\end{equation}
is the quasi-$R$-matrix and
$\D_h = q^{-H\otimes H/2} = \exp(-\frac h4 H\otimes H)$ is the diagonal part of $\R_h$.
Here the quantum integer is denoted $[m]=\frac{q^m-q^{-m}}{q-q^{-1}}$, by abuse of notation. 
Similar to (\ref{e:sD1}, \ref{e:sD2}) one has
\begin{align}
(E^i\otimes 1) \D_h = \D_h (E^i\otimes K^i),\quad &
(1\otimes E^i) \D_h = \D_h (K^i\otimes E^i),\label{D1}\\
(F^i\otimes 1) \D_h = \D_h (F^i\otimes K^{-i}),\quad &
(1\otimes F^i) \D_h = \D_h (K^{-i}\otimes F^i).\label{D2}
\end{align}
One can use Figure \ref{f3} to define the universal invariant $\gh$ of bottom tangles and links associated to $\uh$.
Note that in this case one still uses $K$ and $K^{-1}$ to label the right oriented `cap' and `cup'.

The next version of quantum $\fsl2$ was introduced by Habiro in \cite{habiro5}.
Let $\A=\bbz[q, q^{-1}]$ and $\U$ be the $\A$-subalgebra of $\uh$ generated by
$$
K,\quad K^{-1},\quad e:=(q-q^{-1})E,\quad F^{(n)}:=F^n/[n]!,\quad n\in\bbz, n>0.
$$
Note that $\U\otimes\U$ contains the summands of $\Theta_h$.

The third version of quantum $\fsl2$ is a completion of $\U$:
\begin{equation}\label{tu}
\tu := \left\{\sum_{i=0}^\infty\sum_{j=1}^{N_i} x_{i,j}\in \uh: N_i\ge0,\ x_{i,j}\in\U e^i\U\right\}.
\end{equation}
Then $\tu$ inherits from $\uh$ a complete Hopf algebra structure over $\A$, cf. Section 2 of \cite{habiro4}. The completed tensor product $\tu^{\tt n}$ is the subset of $\uh^{\hat\otimes n}$ consisting of elements of the form
\begin{equation}\label{tun}
\sum_{i=0}^\infty x_{i,1}\otimes\cdots\otimes x_{i,n}
\end{equation}
with $x_{i,j}\in \U$ and at least one $x_{i,j}\in \U e^i\U$, $1\le j\le n$.
%One can consider $\tu$ as a $\tu$-module through the left adjoint action $\trir$.
For any $\tu$-module $W$ let
$$
\Inv(W) := \{w\in W : x(w)=\epsilon(x) w,\ \forall x\in\tu\}
$$
be the invariant submodule.
Now we are ready to prove our key lemma.

\begin{proof}[Proof of Lemma \ref{cen}]
Since $T$ is a bottom tangle with 0 linking matrix, by Theorem 4.1 of \cite{habiro4}
\begin{equation}\label{h}
\gh(T)\in \Inv(\tu^{\tt m}),
\end{equation}
where $\tu^{\tt m}$ is considered as a $\tu$-module through the left adjoint action.
Next we discuss some properties of $\gh(T)$ implied by (\ref{h}).

For any positive integer $n$ let $\U_n := \U/\la e^n, F^{(n)}, K^n-1\ra$ and
$$
\varpi_n : \tu\to\tu/\la e^n\ra = \U/\la e^n\ra \to \U_n
$$
be the composition of the canonical projections, where the equality above is clear from (\ref{tu}).
By (\ref{tu}, \ref{tun}) one has
$$
(\id\otimes\varpi_n^{\otimes (m-1)})(\tu^{\tt m})\subset \tu\otimes\U_n^{\otimes (m-1)}.
$$
Then we have
$$
(\id\otimes\varpi_n^{\otimes (m-1)})(\gh(T))\in \Inv(\tu\otimes\U_n^{\otimes (m-1)}),
$$
thanks to (\ref{h}) and the fact that the ideal $\la e^n, F^{(n)}, K^n-1\ra$ is invariant under $\trir$.
According to Section 7.2 of \cite{habiro1}, $\bar\U_n$ inherits from $\U_n$ a trivial $\tu$-module structure.
Hence for linear maps $\chi_j': \U_n\to\A$ factoring through $\bar\U_n$ we have
$$
\Gamma_n(T;\chi_2',\ldots,\chi_m'):=(\id\otimes\chi_2'\circ\varpi_n\otimes\cdots\otimes\chi_m'\circ\varpi_n)(\gh(T))\in Z(\tu),
$$
where $Z(\tu) = \Inv(\tu)$\footnote{This can be checked easily on the generators.}
denotes the center of $\tu$. For $i\ge0$ let
$$
\sigma_i := \prod_{j=1}^i((q-q^{-1})^2C_h^2-(q^j+q^{-j})^2),
$$
where $C_h = Fe+\frac{qK+q^{-1}K^{-1}}{q-q^{-1}}$ is the Casimir element of $\uh$.
It was shown in \cite{habiro5} that every element in $Z(\tu)$
can be written as
$\sum_{i=0}^\infty z_i \sigma_i$
with $z_i\in \A+(q-q^{-1})\A C_h$. (He also showed in Proposition 9.4 of \cite{habiro5} that $\sigma_i\in \U e^i\U$ so
the infinite sum is in $\tu$.)
Therefore there exist $z_{i}\in \A+(q-q^{-1})\A C_h$ such that
\begin{equation}\label{gn1}
\Gamma_n(T;\chi_2',\ldots,\chi_m') = \sum_{i=0}^\infty z_{i}\sigma_i.
\end{equation}

Now let's compare the calculations of $\jz(T)$ and $\gh(T)$, cf. Figure \ref{f3}. 
One can use (\ref{e:sD1}, \ref{e:sD2}) and (\ref{D1}, \ref{D2}) to slide the diagonal parts $\Dz$ and $\D_h$ respectively to the initial boundary of each component of $T$. 
After the sliding the diagonal parts in $\jz(T)$ and $\gh(T)$ are canceled because $T$ has 0 linking matrix. 
Furthermore it is easily seen from (\ref{e:sD1}, \ref{e:sD2}) and (\ref{D1}, \ref{D2}) that 
the same powers of $K$ are inserted for $\jz(T)$ and $\gh(T)$ after the sliding. Therefore the difference between $\jz(T)$ and $\gh(T)$ is caused by the quasi-$R$-matrices only. For any $\A$-module $W$ let 
$$
\gamma : W\to W\otimes_\A\bbq(\zeta)
$$
be the change of the ground ring map by setting $q=\zeta$.
By comparing (\ref{tz}) and (\ref{th}) we conclude that
$$
\jz(T) = (\gamma\circ \varpi_\ell^{\otimes m})(\gh(T)).
$$
It is clear that there exists $\chi_j': \U_\ell\to \A$ factoring through $\bar\U_\ell$ such that $a_j\chi_j = \chi_j'\otimes \gamma$ for some $a_j\in \bbz[\zeta]$.
It follows that
\begin{equation}\label{T}
(\id\otimes\chi_2\otimes\cdots\otimes\chi_m)(\jz(T)) = \frac1{a_2\cdots a_m}(\gamma\circ(\varpi_\ell\otimes\id^{\otimes(m-1)}))(\Gamma_\ell(T;\chi_2', \ldots, \chi_m')).
\end{equation}
Since (\ref{e:min}) is the minimal polynomial of $C$, cf. \cite{kerler1}, and $\phi(C)$ divides $\gamma\circ\varpi_\ell(\sigma_i)$, $i\ge \ell$, we have
$$
\gamma\circ\varpi_\ell(Z(\tu))\subset \bbq(\zeta)[C].
$$
The lemma now follows from (\ref{T}) and (\ref{gn1}).
\end{proof}

\subsection{\boldmath The integral and trace of elements in $\zz'$}
Recall from Section \ref{wrt} that $\omega = \sum_{i=0}^{\hbar-1} [2i+1] \ve{2i}$.
When restricted to the center $Z(\uz)$, $\trz$ detects $P_i$ only and $\iz$ detects $N_i'$ only, i.e.

\begin{lemma}[\cite{kerler5}]\label{kk}
For $0\le i< \hbar$ and $0\le j\le \hbar$ we have
$$
\trz(N_i) = \trz(N_i') = 0,
$$
\begin{equation}\label{ex}
\iz(P_j) = \iz(N_i) = 0.
\end{equation}
\end{lemma}

\begin{proof}
Recall that $\ve i$ is the irreducible $\uz$-module of dimension $i+1$. Then we have
$$
C|_{\ve i} = \frac{\zeta^{1+i}+\zeta^{-1-i}}{\zeta-\zeta^{-1}}\cdot\id_{\ve i} = b_{i'}\cdot \id_{\ve i},
$$
where $i'\equiv \frac i2 \mod \ell$. This implies $\tr_q^{\ve j} (N_i) = 0$, $j=0,\ldots, \ell-1$, because $N_i$
contains all the factors $(C-b_j)$. (Note that $b_j=b_{\ell-1-j}$.) Therefore $\trz(N_i) = 0$. Similarly we have
$\trz(N_i') = 0$. 
On the other hand since $N_i$ and $P_j$ do not contain the term $E^{\ell-1}K F^{\ell-1}$, $\iz(P_j) = \iz(N_i) = 0$ by (\ref{i}).
(Note that $P_\hbar = \phi_\hbar(C)/\phi_\hbar(b_\hbar)$ because $\phi(C)=0$.)
\end{proof}

The above lemma can be used to show the following relation. 

\begin{lemma}\label{r1}
We have
$$
\iz(\br^{\pm1}) = \pm\frac{\zeta^2}{\ell^2}\,(\zeta-\zeta^{-1})^\ell\, \trz(\br^{\pm1}).
$$
\end{lemma}

\begin{proof}
Since the Casimir element $C$ acts on $\ve{2j}$ as the scalar multiplication by $b_j$ one has
\begin{align}\label{Pq}
\tr_q^{\ve{2j}}(P_i) &= [2j+1] \left(\frac1{\phi_i(b_i)}-\frac{\phi'_i(b_i)}{\phi_i(b_i)^2}(b_j-b_i)\right)\phi_i(b_j) \nonumber\\
&= [2j+1]\,\delta_{i,j}.
\end{align}
This together with (\ref{r}) and Lemma \ref{kk}  gives
\begin{align}\label{aa}
\trz(\br^{\pm1}) &= \sum_{i=0}^\hbar \zeta^{\pm2i(i+1)}\trz(P_i) \nonumber\\
&= \sum_{i=0}^\hbar \zeta^{\pm2i(i+1)}\sum_{j=0}^{\hbar-1}[2j+1]\tr_q^{\ve{2j}}(P_i) \nonumber\\
&= \sum_{i=0}^\hbar \zeta^{\pm2i(i+1)} [2i+1]^2.
\end{align}
Note that $[2\hbar+1]=[\ell]=0$. On the other hand we have
\begin{align}\label{Np}
\iz(N_j') &= \iz (T_jN_j)
= \iz\Big(\frac{T_j}{\phi(b_j)}F^{\ell-1}E^{\ell-1}\Big)\nonumber\\
&= \frac{\zeta^2}{\ell\phi_j(b_j)}\sum_{i=j+1}^{\ell-1-j}\zeta^{2i}
= \frac{\zeta^2[2j+1]^3(\zeta-\zeta^{-1})^\ell}{\ell^3}.
\end{align}
Here we use $\phi_j(b_j) = \frac{\ell^2}{(\zeta-\zeta^{-1})^\ell[2j+1]^2}$. By (\ref{r}) and Lemma \ref{kk} we have
\begin{align}\label{aa1}
\iz(\br^{\pm1}) &= \pm\ell\sum_{i=0}^\hbar \frac{\zeta^{\pm2i(i+1)}}{[2i+1]}\iz (N_i') \nonumber\\
&= \pm \frac{\zeta^2}{\ell^2}(\zeta-\zeta^{-1})^\ell\sum_{i=0}^\hbar \zeta^{\pm2i(i+1)}[2i+1]^2.
\end{align}
The lemma follows from (\ref{aa}) and (\ref{aa1}).
\end{proof}

We introduce the following elements in $\uz^*$:
$$
\trzpm(x) =  \frac{\trz(x)}{\trz(\br^{\pm1})},\quad
\izpm(x) = \frac{\iz(x)}{\iz(\br^{\pm1})}, \quad \forall x \in \uz.
$$

\begin{lemma}\label{10}
For any $x$ in $\zz'$ and non-negative integer $n$,
\begin{equation}\label{eqn4}
\izpm (x\, \br^{\pm n}) \ =\  n\, \trzpm (x\, \br^{\pm n}).
\end{equation}
\end{lemma}

\begin{proof}
By Lemma \ref{r1}, it is enough to show that for any $x$ in $\zz'$ and non-negative integer $n$ we have
\begin{equation}\label{re}
\iz(x \br^{\pm n}) = \pm \frac{n\zeta^2}{\ell^2}(\zeta-\zeta^{-1})^\ell\ \trz(x \br^{\pm n}).
\end{equation}
By Lemma \ref{b} it is enough to show (\ref{re}) for $x = P_i$ and $N_i$.

We show (\ref{re}) for $x=N_i$.
By (\ref{r}), we have
\begin{equation}\label{rn}
\br^{\pm n} = \sum_{i=0}^\hbar \zeta^{\pm 2i(i+1)n}\left(P_i \pm n \frac{2i+1-\ell}{[2i+1]}N_i \pm \frac{n\ell}{[2i+1]}N_i'\right).
\end{equation}
Hence by Lemma \ref{kk}
$$
\trz (N_i\, \br^{\pm n}) = \zeta^{\pm 2i(i+1)n}\, \trz (N_i) = 0,
$$
and
$$
\iz (N_i\,  \br^{\pm n}) = \zeta^{\pm 2i(i+1)n} \iz (N_i) = 0.
$$
Therefore (\ref{re}) holds if $x=N_i$.

We show (\ref{re}) for $x=P_i$.
By (\ref{rn}) 
$$
P_i\, \br^{\pm n} = \zeta^{\pm 2i(i+1)n} \left(P_i \pm n \frac{2i+1-\ell}{[2i+1]}N_i \pm\frac{n\ell}{[2i+1]}N_i'\right).
$$
Hence by Lemma \ref{kk} and (\ref{Pq})
\begin{equation*}
\trz (P_i\,  \br^{\pm n}) = \zeta^{\pm 2i(i+1)n}\, \trz (P_i)
= \zeta^{\pm 2i(i+1)n}\,[2i+1]^2.
\end{equation*}
On the other hand
\begin{align*}
\iz (P_i\,  \br^{\pm n}) &= \frac{\pm n \ell}{[2i+1]} \zeta^{\pm 2i(i+1)n}\, \iz (N'_i)\\
&= \frac{\pm n \ell}{[2i+1]} \zeta^{\pm 2i(i+1)n}\,\iz(N_i')\\
&= \pm \frac{n\zeta^2}{\ell^2}(\zeta-\zeta^{-1})^\ell\, \zeta^{\pm2i(i+1)n} [2i+1]^2 \quad\qquad\text{by (\ref{Np})}\\
&= \pm \frac{n\zeta^2}{\ell^2}(\zeta-\zeta^{-1})^\ell\ \trz(P_i \br^{\pm n}).
\end{align*}
Hence (\ref{re}) holds for $x=P_i$.
\end{proof}

\subsection{Quantum invariants of rational homology 3-spheres}\label{ls}
For any element $c\in \uz$ and $\mu\in \uz^*$, let $\mu^c$ be an element in $\uz^*$ defined by 
$$
(\mu^c)(x) = \mu(c x), \quad \forall x\in\uz.
$$
The following lemma is an easy extension of Remarks \ref{rk} and \ref{rk1}. We omit its proof.

\begin{lemma}\label{ff}
If $c\in Z(\uz)$ then $\trzpm^c$ and $\izpm^c$ factor through $\bar\uz$.
\end{lemma}

Now we are ready to prove Proposition \ref{prop2}.

\begin{proof}[Proof of Proposition \ref{prop2}]
Let $M$ be a rational homology 3-sphere which can be obtained from $S^3$ by surgery along
an algebraically split link $L$. Suppose $L$ has $m$ components with framings $f_1, \ldots, f_m$. Denote the sign of $f_i$ by $\e i$.
Denote by $T_L$ a bottom tangle whose natural closure is $L$.
Let $T_L^{(0)}$ be the bottom  tangle obtained from $T_L$ by changing the framing of each component to 0.
For any $1\le i\le m$ set
$$
z_i := \left(\bigotimes_{j=1}^{i-1}\trze{j}^{\br^{-f_j}}\otimes\id\otimes
\bigotimes_{k=i+1}^m\ize{k}^{\br^{-f_{k}}}\right)(\jz(T^{(0)})) 
$$
By Lemmas \ref{cen} and \ref{ff} we have $z_i\in\zz'$. Therefore by Lemma \ref{10} we have
\begin{equation}\label{L}
\ize i^{\br^{-f_i}} (z_i) = |f_i|\cdot \trze i^{\br^{-f_i}} (z_i).
\end{equation}
This equation shows that switching from $\ize i^{\br^{-f_i}}$ to $\trze i^{\br^{-f_i}}$ on the $i$-th component produces a factor $|f_i|$.
Hence by (\ref{hh}) we have
\begin{align*}
\pz(M) &= \left(\ize1\otimes\cdots\otimes\ize m\right)(\jz(L))\\
&=\left(\ize1^{\br^{-f_1}}\otimes\cdots\otimes\ize m^{\br^{-f_m}}\right)(\jz(T^{(0)})) \\
&=\ize1^{\br^{-f_1}}(z_1) = |f_1|\cdot\trze1^{\br^{-f_1}}(z_1) \hspace{5ex} \text{by (\ref{L})}\\
&= |f_1|\cdot\ize2^{\br^{-f_2}}(z_2) = |f_1f_2|\cdot\trze2^{\br^{-f_2}}(z_2)=\cdots \\
%&= |f_1f_2|\cdot\ize3^{\br^{-f_3}}(z_3)=\cdots \\
&= |f_1\cdots f_m| \cdot\trze{m}^{\br^{-f_m}}(z_m) \\
&= |f_1\cdots f_m|\cdot (\trze1^{\br^{-f_1}}\otimes\cdots\otimes\trze m^{\br^{-f_m}})(\jz(T^{(0)}))\\
&= |f_1\cdots f_m|\cdot (\trze1\otimes\cdots\otimes\trze m)(\jz(L)) \\
&= \h(M)\,\tz(M).
\end{align*}
This ends the proof of Proposition \ref{prop2}.
\end{proof}

\end{document}